\documentclass[11pt]{amsart}
\usepackage{amsfonts,amssymb,amsthm}
\usepackage{amsmath,amscd}
\usepackage{pstricks}
\usepackage{pstricks,pst-node}
\usepackage{mathrsfs}
\usepackage[all]{xy}

\DeclareMathAlphabet{\mathpzc}{OT1}{pzc}{m}{it}

\renewcommand{\subsection}[1]{\vspace{.18in}
\par\noindent\addtocounter{subsection}{1}
\setcounter{equation}{0}{\bf\thesubsection.\hspace{5pt}#1}}

\theoremstyle{definition}
\newtheorem{Rem}[subsection]{Remark}

\theoremstyle{plain}
\newtheorem{Prop}[subsection]{Proposition}
\newtheorem{Thm}[subsection]{Theorem}
\newtheorem{Not}[subsection]{Notation}
\newtheorem{Lem}[subsection]{Lemma}
\newtheorem{Coro}[subsection]{Corollary}
\numberwithin{equation}{subsection}

\newcommand{\bse}{\boldsymbol{e}}

\newcommand{\bfd}{{\mathbf{d}}}

\newcommand{\bfj}{{\mathbf{j}}}

\newcommand{\bfl}{{\mathbf{0}}}

\newcommand{\bfx}{{\mathbf{x}}}

\newcommand{\bfU}{{\mathbf{U}}}

\def\fS{{\frak S}}

\def\fka{{\frak a}}
\def\fkd{{\frak d}}

\def\sI{{\mathcal I}}
\def\sJ{{\mathcal J}}

\def\sU{{\mathcal U}}

\def\sZ{{\mathcal Z}}

\newcommand{\mbn}{\mathbb N}
\newcommand{\mbq}{\mathbb Q}
\newcommand{\mbc}{\mathbb C}
\newcommand{\mbz}{\mathbb Z}

\newcommand{\Aut}{\operatorname{Aut}}
\newcommand{\End}{\operatorname{End}}

\newcommand{\spann}{\operatorname{span}}
\newcommand{\diag}{\operatorname{diag}}

\def\ro{\text{\rm ro}}
\def\co{\text{\rm co}}

\newcommand{\la}{{\lambda}}
\newcommand{\La}{\Lambda}
\newcommand{\ga}{{\gamma}}

\newcommand{\dt}{\delta}

\newcommand{\up}{\upsilon}
\newcommand{\vi}{\varphi}

\newcommand{\al}{\alpha}
\newcommand{\bt}{\beta}
\newcommand{\sg}{\sigma}

\def\th{\theta} 

\newcommand{\lan}{\langle}
\newcommand{\ran}{\rangle}

\newcommand{\leb}{\left[}

\newcommand{\rib}{\right]}
\def\lr#1{\langle #1\rangle}

\def\bpa#1#2{\left({#1\atop #2}\right)}

\def\ggp#1#2{\left[\kern-3.2pt\left[{#1\atop #2}\right]\kern-3.2pt\right]}

\newcommand{\p}{\prec}
\newcommand{\pr}{\preccurlyeq}
\def\leq{\leqslant}\def\geq{\geqslant}
\def\le{\leqslant}\def\ge{\geqslant}

\newcommand{\bop}{\bigoplus}

\newcommand{\ot}{\otimes}

\newcommand{\han}{\subseteq}

\newcommand{\h}{\widehat}
\newcommand{\ti}{\widetilde}
\newcommand{\tu}{\widetilde u}

\newcommand{\tr}{{}^t\!}
\newcommand{\tA}{{}^t\!A}
\newcommand{\tB}{{}^t\!B}

\newcommand{\lra}{\longrightarrow}
\newcommand{\ra}{\rightarrow}
\newcommand{\map}{\mapsto}

\newcommand{\dzr}{\dot\zeta_r}

\newcommand{\zr}{\zeta_r}

\newcommand{\vtg}{{\!\vartriangle\!}}
\newcommand{\Ha}{{{\mathfrak H}_\vtg(n)}}

\newcommand{\dbfHa}{{\boldsymbol{\mathfrak D}_\vtg}(n)}
\newcommand{\dbfHap}{{\boldsymbol{\mathfrak D}^+_\vtg}(n)}
\newcommand{\dbfHam}{{\boldsymbol{\mathfrak D}^-_\vtg}(n)}
\newcommand{\dbfHaz}{{\boldsymbol{\mathfrak D}^0_\vtg}(n)}
\newcommand{\dHa}{{{\mathfrak D}_\vtg}(n)}
\newcommand{\dHap}{{{\mathfrak D}^+_\vtg}(n)}
\newcommand{\dHam}{{{\mathfrak D}^-_\vtg}(n)}
\newcommand{\dHaz}{{{\mathfrak D}^0_\vtg}(n)}

\newcommand{\ddbfHa}{\dot{\boldsymbol{\mathfrak D}_\vtg}(n)}
\newcommand{\hddbfHa}{\widehat{\dot{\boldsymbol{\mathfrak D}}}_\vtg(n)}
\newcommand{\ddHa}{\dot{{\mathfrak D}_\vtg}(n)}
\newcommand{\hddHa}{\widehat{\dot{{\mathfrak D}}}_\vtg(n)}

\newcommand{\afuglq}{\sU(\widehat{\frak{gl}}_n)}

\newcommand{\afuglqz}{\sU^0(\widehat{\frak{gl}}_n)}
\newcommand{\afuglqp}{\sU^+(\widehat{\frak{gl}}_n)}

\newcommand{\afuglqm}{\sU^-(\widehat{\frak{gl}}_n)}
\newcommand{\afuglz}{\sU_\mbz(\widehat{\frak{gl}}_n)}
\newcommand{\afuglzp}{\sU_\mbz^+(\widehat{\frak{gl}}_n)}
\newcommand{\afuglzm}{\sU_\mbz^-(\widehat{\frak{gl}}_n)}
\newcommand{\afuglzz}{\sU_\mbz^0(\widehat{\frak{gl}}_n)}

\newcommand{\afSrmbz}{{\mathcal S}_{\vtg}(n,r)_{\mathbb Z}}
\newcommand{\afSrmbq}{{\mathcal S}_{\vtg}(n,r)_{\mathbb Q}}

\newcommand{\ddHambq}{\dot{{\mathfrak D}_\vtg}(n)_\mbq}
\newcommand{\hddHambq}{\widehat{\dot{{\mathfrak D}}}_\vtg(n)_\mbq}
\newcommand{\ddHambz}{\dot{{\mathfrak D}_\vtg}(n)_\mbz}
\newcommand{\hddHambz}{\widehat{\dot{{\mathfrak D}}}_\vtg(n)_\mbz}

\newcommand{\dHapmbz}{{{\mathfrak D}^+_\vtg}(n)_\mbz}
\newcommand{\dHammbz}{{{\mathfrak D}^-_\vtg}(n)_\mbz}
\newcommand{\dHapmbq}{{{\mathfrak D}^+_\vtg}(n)_\mbq}
\newcommand{\dHammbq}{{{\mathfrak D}^-_\vtg}(n)_\mbq}
\newcommand{\dHapmmbq}{{{\mathfrak D}^\pm_\vtg}(n)_\mbq}

\newcommand{\afgl}{\widehat{\frak{gl}}_n}

\newcommand{\afal}{{\boldsymbol\alpha}^\vartriangle}

\newcommand{\afbse}{\boldsymbol e^\vartriangle}
\newcommand{\afPin}{\Pi_\vtg(n)}

\newcommand{\afE}{E^\vartriangle}

\newcommand{\Kbfj}{K^{\bfj}}
\newcommand{\su}{{}^{}}
\def\su#1{^{#1}}

\newcommand{\afbfU}{\bfU_\vtg(n)}

\newcommand{\afSr}{{\mathcal S}_{\vtg}(n,r)}

\newcommand{\afbfSr}{{\boldsymbol{\mathcal S}}_\vtg(n,r)}

\newcommand{\afmbnn}{\mathbb N_\vtg^{n}}
\newcommand{\afmbzn}{\mathbb Z_\vtg^{n}}

\newcommand{\afThn}{\Theta_\vtg(n)}

\newcommand{\afThnpm}{\Theta_\vtg^\pm(n)}
\newcommand{\afThnp}{\Theta_\vtg^+(n)}
\newcommand{\afThnm}{\Theta_\vtg^-(n)}

\newcommand{\afThnr}{\Theta_\vtg(n,r)}

\newcommand{\afMnq}{M_{\vtg,n}(\mathbb Q)}

\newcommand{\afLanr}{\Lambda_\vtg(n,r)}
\newcommand{\afLann}{\Lambda_\vtg(n,n)}

\newcommand{\tri}{\triangle(n)}
\def\field{{\mathbb F}}

\begin{document}
\title{Integral affine Schur--Weyl reciprocity}
\author{Qiang Fu}
\address{Department of Mathematics, Tongji University, Shanghai, 200092, China.}
\email{q.fu@hotmail.com}


\thanks{Supported by the National Natural Science Foundation
of China, the Program NCET, Fok Ying Tung Education Foundation
and the Fundamental Research Funds for the Central Universities}

\begin{abstract}
Let $\dbfHa$ be the double Ringel--Hall algebra of the cyclic quiver $\tri$ and let $\ddbfHa$ be the modified quantum affine algebra of $\dbfHa$. We will construct an integral form $\ddHa$ for $\ddbfHa$ such that the natural algebra homomorphism from $\ddHa$ to the integral affine quantum Schur algebra is surjective. Furthermore, we will use Hall algebras to construct the integral form $\afuglz$ of the universal enveloping algebra $\afuglq$ of the loop algebra $\afgl=\frak{gl}_n(\mbq)\ot\mbq[t,t^{-1}]$, and prove that the natural algebra homomorphism from $\afuglz$ to the affine Schur algebra over $\mbz$ is surjective.
\end{abstract}
 \sloppy \maketitle
\section{Introduction}

The representation of the general linear group $GL(n,\mbc)$
and the symmetric group $\fS_r$ over $\mbc$ are related by Schur--Weyl reciprocity (cf. \cite{Weyl}). This reciprocity is also true over $\mbz$. That is, the natural algebra homomorphisms
$$\sU_\mbz(\frak{gl}_n)\ra\End_{\mbz[\fS_r]}(V^{\ot r}),\ \mbz[\fS_r]\ra\End_{\sU_\mbz(\frak{gl}_n)}(V^{\ot r})$$
are surjective, where $\sU_\mbz(\frak{gl}_n)$ is the Kostant $\mbz$-form \cite{Ko} of the universal enveloping algebra $\sU(\frak{gl}_n)$ of $\frak{gl}_n:=\frak{gl}_n(\mbq)$, and $V$ is the natural module for $\sU_\mbz(\frak{gl}_n)$ (see \cite{CL,CP,Donkin}).
The quantum Schur--Weyl reciprocity between quantum $\frak{gl}_n$ and Hecke algebras of type $A$ in the generic case was established in \cite{Jimbo} and the integral quantum Schur--Weyl reciprocity was proved in \cite{Du95,DPS}. Furthermore, the cyclotomic Schur--Weyl reciprocity between quantum groups and Ariki--Koike algebras was investigated in \cite{ATY,SS,Ar,Hu}.

Let $\dbfHa$ be the double Ringel--Hall algebra of the cyclic quiver $\tri$ over $\mbq(\up)$,
where $\up$ is
an indeterminate. Then $\dbfHa$ is isomorphic to the quantum loop algebra $\bfU(\afgl)$ defined by Drinfeld's new presentation (cf. \cite[2.3.5]{DDF}), where $\afgl=\frak{gl}_n\ot\mbq[t,t^{-1}]$ is the loop algebra associated to $\frak{gl}_n$. In \cite[3.6.3]{DDF}, it is proved that the natural algebra homomorphism $\zeta_r$ from $\dbfHa$ to $\afbfSr$ is surjective, where
$\afbfSr$ is the affine quantum Schur algebra over
$\mbq(\up)$ (with $\up$ an indeterminate). It is natural to ask whether this result is true over any field. Before discussing this problem, we have to construct a suitable integral form for $\dbfHa$. Let $\sZ=\mbz[\up,\up^{-1}]$. A certain $\sZ$-submodule of $\dbfHa$, denoted by $\dHa$, was introduced in \cite[(3.8.1.1)]{DDF}, and it is conjectured in \cite[3.8.6]{DDF} that $\dHa$ is a $\sZ$-subalgebra of $\dbfHa$. If this conjecture is true, then $\dHa$ becomes an integral form for $\dbfHa$.

Let $\ddbfHa$ be the modified quantum affine algebra of $\dbfHa$. Associated with $\dHa$, we will construct a certain free $\sZ$-submodule of $\ddbfHa$, denoted by $\ddHa$, such that $\ddbfHa=\ddHa\ot_\sZ\mbq(\up)$. We will prove in \ref{interal form} and \ref{integral quantum affine Schur Weyl duality} that $\ddHa$ is a $\sZ$-subalgebra of $\ddbfHa$ and the natural algebra homomorphism $\dzr$ from $\ddHa$ to $\afSr$ is surjective, where $\afSr$ is the affine quantum Schur algebra over $\sZ$.

Let $\hddbfHa$ (resp., $\hddHa$) be the completion algebra of $\ddbfHa$ (resp., $\ddHa$). We will see in \ref{Phi} and \ref{Rem} that the double Ringel--Hall algebra $\dbfHa$ can be regarded as a subalgebra of $\hddbfHa$ and we have a proper inclusion $\dHa\subset\hddHa\cap\dbfHa$.
Furthermore we will prove in \ref{v=1} that this proper inclusion becomes an equality in the classical case.
More precisely, we will use Hall algebras to construct a certain lattice, denoted by $\afuglz$, of the universal enveloping algebra $\afuglq$ of $\afgl$.
Let $\hddHambq$ (resp., $\hddHambz$) be the completion algebra of $\ddHa\ot_\sZ\mbq$ (resp., $\ddHa\ot_\sZ\mbz$), where $\mbq$ and $\mbz$ are regarded as $\sZ$-modules by specializing $\up$ to $1$.
We will prove in \ref{v=1} that $\afuglz=\hddHambz\cap\afuglq$. Here
$\afuglq$ is regarded as a subalgebra of $\hddHambq$ via the map $\vi$ defined in \ref{vi}. In particular, we conclude that $\afuglz$ is a $\mbz$-subalgebra of $\afuglq$ and hence $\afuglz$ is the integral form of $\afuglq$. As the quantum affine case, there is a natural surjective algebra homomorphism $\eta_r:\afuglq\ra\afSrmbq$, where $\afSrmbq=\afSr\ot_\sZ\mbq$. We will prove in \ref{integral affine Schur Weyl duality} that
the restriction of $\eta_r$ to $\afuglz$ yields a surjective $\mbz$-algebra homomorphism $\eta_r:\afuglz\ra\afSrmbz$, where $\afSr_\mbz=\afSr\ot_\sZ\mbz$ (cf. \cite{Fu}).

We organize this paper as follows. In \S2, we will recall the definition of the double Ringel--Hall algebra $\dbfHa$ and the modified quantum affine algebra $\ddbfHa$ of $\dbfHa$. We collect in \S3 several results concerning affine quantum Schur algebras. In \S4 we will construct the $\sZ$-submodule $\dHa$ (resp., $\ddHa$) of $\dbfHa$ (resp., $\ddbfHa$) and prove in \ref{interal form} that $\ddHa$ is a $\sZ$-subalgebra of $\ddbfHa$. In addition, we will
prove in \ref{integral quantum affine Schur Weyl duality} that the natural algebra homomorphism $\dzr$ from $\ddHa$ to the integral quantum affine Schur algebra $\afSr$ is surjective and
establish certain relation between $\dHa$ and $\ddHa$ in \ref{Rem}.
In \ref{commuting formula between elements associated with indecomposable modules}, we derive certain commutator formulas in $\dbfHa$, which will be used in \S6. Finally, we will use Hall algebras to introduce the free $\mbz$-submodule $\afuglz$ of $\afuglq$, and prove in \ref{v=1} and \ref{integral affine Schur Weyl duality} that $\afuglz$ is a $\mbz$-subalgebra of $\afuglq$ such that the natural algebra homomorphism $\eta_r:\afuglz\ra\afSrmbz$ is surjective.

\begin{Not}\label{Notaion} \rm
For a positive integer $n$, let
$\afMnq$  be the set of all matrices
$A=(a_{i,j})_{i,j\in\mbz}$ with $a_{i,j}\in\mbq$ such that
\begin{itemize}
\item[(a)]$a_{i,j}=a_{i+n,j+n}$ for $i,j\in\mbz$; \item[(b)] for
every $i\in\mbz$, both sets $\{j\in\mbz\mid a_{i,j}\not=0\}$ and
$\{j\in\mbz\mid a_{j,i}\not=0\}$ are finite.
\end{itemize}
Let
$\afThn=\{A\in\afMnq\mid a_{i,j}\in\mbn,\,\forall i,j\}$.

Let $\afmbzn=\{(\la_i)_{i\in\mbz}\mid
\la_i\in\mbz,\,\la_i=\la_{i-n}\ \text{for}\ i\in\mbz\}\text{ \,\,
and \,\,} \afmbnn=\{(\la_i)_{i\in\mbz}\in \afmbzn\mid \la_i\ge0\text{ for  }i\in\mbz\}.$
We will identify $\afmbzn$ with $\mbz^n$ via the following bijection
\begin{equation}\label{flat}
\flat:\afmbzn\lra\mbz^n,\quad \bfj\longmapsto \flat(\bfj)=(j_1,\cdots,j_n).
\end{equation}

Let $\sZ=\mbz[\up,\up^{-1}]$, where $\up$ is an indeterminate, and let $\mbq(\up)$ be the fraction field of $\sZ$. Specializing $\up$ to $1$, $\mbq$ and
$\mbz$ will be viewed as $\sZ$-modules.

\end{Not}

\section{Double Ringel--Hall algebras of cyclic quivers}

Let $\tri$ ($n\geq 2$) be
the cyclic quiver
with vertex set $I=\mbz/n\mbz=\{1,2,\ldots,n\}$ and arrow set
$\{i\to i+1\mid i\in I\}$. Let $\field$ be a field. For $i\in I$, let $S_i$
be the irreducible representation of $\tri$ over $\field$ with $(S_i)_i=\field$ and $(S_i)_j=0$ for $i\neq j$.
Let $$\afThnp:=\{A\in\afThn\mid a_{i,j}=0\text{ for }i\geq j\}.$$
For any $A=(a_{i,j})\in\afThnp$, let
$$M(A)=M_\field(A)=\bop_{1\leq i\leq n\atop i<j,\,j\in\mbz}a_{i,j}M^{i,j},$$
where
$M^{i,j}$ is the unique indecomposable representation for $\tri$ of length $j-i$ with top $S_i$.
For $A\in\afThnp$ let $\bfd(A)\in\mbn I$ be the dimension vector of $M(A)$. We will identify $\mbn I$ with $\afmbnn$ under \eqref{flat}. By definition we have
\begin{equation}\label{bfd(A)}
\bfd(A)=\bigg(\sum_{s\leq i<t\atop s,t\in\mbz}a_{s,t}\bigg)_{i\in\mbz}
\end{equation}
for $A\in\afThnp$,

For $i,j\in\mbz$ let $\afE_{i,j}\in\afThn$ be the matrix
$(e^{i,j}_{k,l})_{k,l\in\mbz}$ defined by
\begin{equation*}e_{k,l}^{i,j}=
\begin{cases}1&\text{if $k=i+sn,l=j+sn$ for some $s\in\mbz$,}\\
0&\text{otherwise}.\end{cases}
\end{equation*}
For $\la\in\afmbnn$ let $$A_\la=\sum_{1\leq i\leq n}\la_i\afE_{i,i+1}\in\afThnp.$$ Then $M_{\field}(A_\la)$ is a semisimple representation of $\tri$ over $\field$.

The Euler form associated with the cyclic quiver $\tri$ is the
bilinear form $\lan-,-\ran$: $\afmbzn\times\afmbzn\ra\mbz$ defined by
$\lan\la,\mu\ran=\sum_{1\leq i\leq n}\la_i\mu_i-\sum_{1\leq i\leq n}\la_i\mu_{i+1}$
for $\la,\mu\in\afmbzn$.

By \cite{Ri93}, for $A,B,C\in\afThnp$,
there is a polynomial $\vi^{C}_{A,B}\in\mbz[\up^2]$  such
that, for any finite field $\field_q$,
$\vi^{C}_{A,B}|_{\up^2=q}$ is equal to the number of submodules $N$ of
$M_{\field_q}(C)$ satisfying $N\cong M_{\field_q}(B)$ and $M_{\field_q}(C)/N\cong M_{\field_q}(A)$.

Let $\dbfHa$ be the double Ringel--Hall algebra of the cyclic quiver of $\tri$ (cf. \cite{X97} and \cite[(2.1.3.2)]{DDF}). By \cite[2.4.1 and 2.4.4 and 3.9.2]{DDF} we obtain the following.
\begin{Lem} \label{presentation-dbfHa}
The algebra $\dbfHa$ is the algebra over $\mbq(\up)$ generated by
$u_A^+$, $K_{i}^{\pm 1}$, $u_A^-$ $(A\in\afThnp,\,i\in\mbz)$ subject to
the following relations:
\begin{itemize}
\item[(1)]
$K_i=K_{i+n}$, $K_iK_j=K_jK_i$, $K_iK_i^{-1}=1$, $u_0^+=u_0^-=1$;
\item[(2)]
$K\su{\bfj} u_A^+=\up^{\lr{\bfd(A),\bfj}}u_A^+K\su\bfj$,
$u_A^-K\su\bfj=\up^{\lr{\bfd(A),\bfj}}K\su\bfj u_A^-$, where
$K\su\bfj=K_1^{j_1}\cdots K_n^{j_n}$ for $\bfj\in\afmbzn$;
\item[(3)]
$u_A^+u_B^+=\sum_{C\in\afThnp}\up^{\lan \bfd(A),\bfd(B)\ran}\vi_{A,B}^C u_C^+$;
\item[(4)]
$u_A^-u_B^-=\sum_{C\in\afThnp}\up^{\lan \bfd(B),\bfd(A)\ran}\vi_{B,A}^C u_C^-$;
\item[(5)] {\rm commutator relations}:  for all $\la,\mu\in\afmbnn$,
\begin{equation*}
\aligned
\up^{\lan\mu,\mu\ran}&\sum_{\al,\bt\in\afmbnn\atop\la-\al=\mu-\bt\geq 0}\vi_{\la,\mu}^{\al,\bt}
\up^{\lan \bt,\la+\mu-\bt\ran}\ti K\su{\mu-\bt}u_{A_\bt}^-u_{A_\al}^+
=\up^{\lan\mu,\la\ran}\sum_{\al,\bt\in\afmbnn\atop\la-\al=\mu-\bt\geq 0}
{\vi_{\la,\mu}^{\al,\bt}}\up^{\lan \mu-\bt,\al\ran+\lan \mu,\bt\ran}
\ti K\su{\bt-\mu}u_{A_\al}^+u_{A_\bt}^-,\endaligned
\end{equation*}
\end{itemize}
where $\ti K\su\nu :=(\ti K_1)^{\nu_1}\cdots(\ti K_n)^{\nu_n}$ with $\ti K_i=K_iK_{i+1}^{-1}$ for $\nu\in\afmbzn$, and
\begin{equation*}
\vi_{\la,\mu}^{\al,\bt}=\up^{2\sum_{1\leq i\leq n}(\la_i-\al_i)(1-\al_i-\bt_i)}\prod_{1\leq i\leq n\atop 0\leq s\leq\la_i-\al_i-1}\frac{1}{\up^{2(\la_i-\al_i)}-\up^{2s}}.
\end{equation*}
\end{Lem}

Let $\afbfU$ be the subalgebra of $\dbfHa$ generated by $u^+_{\afE_{i,i+1}}$, $u^-_{\afE_{i+1,i}}$ and $K_i^{\pm 1}$ for $1\leq i\leq n$. The algebra $\afbfU$ is a proper subalgebra of $\dbfHa$ and it is the quantum affine algebra considered in \cite[7.7]{Lu99}.

Let $\dbfHap=\spann_{\mbq(\up)}\{u_A^+\mid A\in\afThnp\}$, $\dbfHam=\spann_{\mbq(\up)}\{u_A^-\mid A\in\afThnp\}$, and
$\dbfHaz=\spann_{\mbq(\up)}\{\Kbfj\mid \bfj\in\afmbzn\}$. Then we have
\begin{equation}\label{tri deco of dbfHa}
\dbfHa\cong\dbfHap\ot\dbfHaz\ot\dbfHam.
\end{equation}

For $i\in\mbz$ let
$\afbse_i\in\afmbnn$ be the element
satisfying $\flat(\afbse_i)=\bse_i=(0,\cdots,0,\underset
{(i)}1,0,\cdots,0)$, where $\flat$ is defined in \eqref{flat}.
Let $\afPin=\{\afal_j:=\afbse_j-\afbse_{j+1}\mid 1\leq j\leq n\}$. By \ref{presentation-dbfHa} the algebra $\dbfHa$ is a $\afmbzn$-graded algebra with
$$\deg(u_A^+)=\sum\limits_{1\leq i\leq n}d_i\afal_i,\
\deg(u_A^-)=-
\sum_{1\leq i\leq n}d_i\afal_i\ \text{and}\ \deg(K_i^{\pm 1})=0$$
for $A\in\afThnp$ and $1\leq i\leq n$, where $(d_i)_{i\in\mbz}=\bfd(A)$.
For $\nu\in\afmbzn$ let $\dbfHa_\nu$ be the set of homogeneous elements in $\dbfHa$ of degree $\nu$. Then we have
$$\dbfHa=\bop_{\nu\in\mbz\afPin}\dbfHa_\nu.$$
\begin{Lem}\label{commuting formula K^bfj t}
For $\la,\bfj\in\afmbzn$ and $t\in\dbfHa_\la$ we have
$\Kbfj t=\up^{\bfj\cdot\la}t \Kbfj$,
where $\la\cdot\bfj=\sum_{1\leq j\leq n}\la_ij_i$.
\end{Lem}
\begin{proof}
Clearly for $A\in\afThnp$ we have $\lan\bfd(A),\bfj\ran=(\deg u_A^+)\cdot\bfj$ and  $-\lan\bfd(A),\bfj\ran=(\deg u_A^-)\cdot\bfj$. Combining this with \ref{presentation-dbfHa}(2) proves the assertion.
\end{proof}

Following \cite{Lubk} we now introduce the modified quantum affine algebra $\ddbfHa$ of $\dbfHa$.
For $\la,\mu\in\afmbzn$ we set
$${}_\la\dbfHa_\mu=\dbfHa\bigg/\bigg(\sum_{\bfj\in\afmbzn}(\Kbfj-
 \up^{\la\cdot\bfj})\dbfHa+\sum_{\bfj\in\afmbzn}\dbfHa(\Kbfj
 -\up^{\mu\cdot\bfj})\bigg).$$
Let $\pi_{\la,\mu}:\dbfHa\ra{}_\la\dbfHa_\mu$ be the canonical projection. Let  $$\ddbfHa:=\bop\limits_{\la,\mu\in\afmbzn}{}_\la\dbfHa_\mu.$$
Since ${}_\la\dbfHa_\mu=\bop_{\nu\in\mbz\afPin}\pi_{\la,\mu}(\dbfHa_\nu)$, we have $\ddbfHa=\bop_{\nu\in\mbz\afPin}\ddbfHa_\nu$, where $\ddbfHa_\nu=\bop_{\la,\mu\in\afmbzn}\pi_{\la,\mu}(\dbfHa_\nu)$.

\begin{Lem}\label{Lem for def of ddbfHa}
Assume $\la,\mu\in\afmbzn$, $\nu\in\mbz\afPin$ and $\nu\not=\la-\mu$. Then we have $\pi_{\la,\mu}(\dbfHa_\nu)=0$.
\end{Lem}
\begin{proof}
Let $t\in\dbfHa_\nu$. By \ref{commuting formula K^bfj t} we see that
$(\up^{\la_i}-\up^{\mu_i+\nu_i})\pi_{\la,\mu}(t)
=\pi_{\la,\mu}(K_i
t)-\pi_{\la,\mu}
(\up^{\nu_i}tK_i)=0$
for $1\leq i\leq n$. Since $\nu\not=\la-\mu$ and $\up$ is an indeterminate, there exist
$1\leq i_0\leq n$ such that $\up^{\la_{i_0}}\not=\up^{\mu_{i_0}+\nu_{i_0}}$. Consequently,  $\pi_{\la,\mu}(t)=0$.
\end{proof}

We define the product in $\ddbfHa$ as follows.
For $\la',\mu',\la'',\mu''\in\afmbzn$ with
$\la'-\mu',\la''-\mu''\in\mbz\afPin$ and any $t\in\dbfHa_{\la'-\mu'}$,
$s\in\dbfHa_{\la''-\mu''}$,  define
$$\pi_{\la',\mu'}(t)\pi_{\la'',\mu''}(s)=\begin{cases}\pi_{\la',\mu''}(ts),
& \text{if } \mu'=\la''\\
0& \text{otherwise}.
\end{cases}$$
Then by \ref{Lem for def of ddbfHa} one can check that
$\ddbfHa$ becomes an associative $\mbq(\up)$-algebra structure with the above product.

The algebra $\ddbfHa$ is naturally a $\dbfHa$-bimodule defined by
\begin{equation}\label{bimodule}
t'\pi_{\la',\la''}(s)t''=\pi_{\la'+\nu',\la''-\nu''}(t'st'') \end{equation}
for $t'\in\dbfHa_{\nu'}$, $s\in\dbfHa$, $t''\in\dbfHa_{\nu''}$ and $\la',\la''\in\afmbzn$.

\section{Affine quantum Schur algebras}

For $r\geq 0$ let $\afbfSr$\footnote{The
algebra $\afbfSr$ is denoted by  $\frak U_{r,n,n}$ in \cite[1.9]{Lu99}.} be the affine quantum Schur algebra over $\mbq(\up)$ defined in \cite[1.9]{Lu99}.
Recall the set $\afThn$ defined in \ref{Notaion}.
The algebra $\afbfSr$ has a normalized  $\mbq(\up)$-basis $\{[A]\mid A\in\afThnr\}$ (cf. \cite[1.9]{Lu99}),
where $$
\afThnr=\{A\in\afThn\mid\sg(A):=\sum_{1\leq i\leq n\atop
j\in\mbz}a_{i,j}=r\}.$$
Let $\sZ=\mbz[\up,\up^{-1}]$, where $\up$ is an indeterminate. Let $\afSr$ be the $\sZ$-submodule of $\afbfSr$ spanned by $\{[A]\mid A\in\afThnr\}$. Then $\afSr$ is the $\sZ$-subalgebra of $\afbfSr$.

For $r\geq 0$, let
$\afLanr=\{\la\in\afmbnn\mid\sg(\la):=\sum_{1\leq i\leq n}\la_i=r\}.$
For $\la\in\afLanr$ and $A\in\afThnr$, we have
\begin{equation}\label{[diag(la)][A]}
\begin{aligned}
\ [\diag(\la)]\cdot[A]=
\begin{cases}[A], & \text{if}\ \lambda=\ro(A);\\
0, & \text{otherwise;}
\end{cases}
\end{aligned} \ \text{and}\
\begin{aligned}
\ [A][\diag(\la)]=
\begin{cases}[A], & \text{if}\ \lambda=\co(A);\\
0, & \text{otherwise,}
\end{cases}
\end{aligned}
\end{equation}
where
$\ro(A)=\bigl(\sum_{j\in\mbz}a_{i,j}\bigr)_{i\in\mbz}$ and $\co(A)=\bigl(\sum_{i\in\mbz}a_{i,j}\bigr)_{j\in\mbz}$ (see \cite[1.9]{Lu99}). In particular, we have
\begin{equation}\label{[diag(la)][diag(mu)]}
[\diag(\la)][\diag(\mu)]=\dt_{\la,\mu}[\diag(\la)]
\end{equation}
for $\la,\mu\in\afLanr$.

We now recall certain triangular relation in $\afbfSr$, which will be needed in \S3. First we need the following order relation $\pr$
on $\afThn$.
For $A\in\afThn$ and $i\not=j\in\mbz$, let
$$\sg_{i,j}(A)=\sum\limits_{s\leq i,t\geq j}a_{s,t}\text{ if $i<j$,}\text{ and }
\sg_{i,j}(A)=\sum\limits_{s\geq i,t\leq j}a_{s,t}  \text{ if
$i>j$}.$$ For $A,B\in\afThn$, define
$B\pr A$ if $\sg_{i,j}(B)\leq\sg_{i,j}(A)$ for all $i\not=j$.
Put $B\p A$ if $B\pr A$ and, for some pair $(i,j)$ with $i\not=j$,
$\sg_{i,j}(B)<\sg_{i,j}(A)$.

Let $\afThnpm:=\{A\in\afThn\mid a_{i,j}=0\text{ for }i= j\}$.
For $A\in\afThnpm$ and $\bfj\in\afmbzn$, define $A(\bfj,r)\in
\afbfSr$ by
\begin{equation*}
A(\bfj,r)=\begin{cases}
\sum_{\la\in\La_\vtg(n,r-\sg(A))}\up^{\la\cdot\bfj}[A+\diag(\la)],&\text{ if }\sg(A)\leq r;\\
0,&\text{ otherwise.}\end{cases}
\end{equation*}
For $A\in\afThnpm$, write $A=A^++A^-$ with $A^+\in\afThnp$,
$A^-\in\afThnm$, where $$\afThnm:=\{A\in\afThn\mid a_{i,j}=0\text{ for }i\leq j\}.$$ The following triangular relation in $\afbfSr$ is given in \cite[3.7.3]{DDF}.

\begin{Lem}\label{tri}
Let $C\in\afThnpm$. Then the following triangular relation holds in $\afbfSr$:
$$
C^+(\bfl,r)C^-(\bfl,r)=C(\bfl,r)+\sum_{X\in\afThnpm\atop X\p C,\,\bfj\in\afmbzn}h_{C,X,\bfj;r}X(\bfj,r),
$$
where $h_{C,X,\bfj;r}\in\mbq(\up)$.
\end{Lem}
Using \ref{tri} one can construct a  $\sZ$-basis for $\afSr$ as follows.
\begin{Coro}[{\cite[3.7.7]{DDF}}]\label{integral basis for affine Schur algebras}
The set $\{A^+(\bfl,r)[\diag(\la)]A^-(\bfl,r)\mid A\in\afThnpm,\,\la\in\afLanr,\,\la_i\geq\sg_i(A),\,\text{for}\,1\leq i\leq n\}$ forms a $\sZ$-basis for $\afSr$, where $\sg_i(A)=\sum_{j<i}(a_{i,j}+a_{j,i})$.
\end{Coro}

The double Ringel--Hall algebra $\dbfHa$ is related to the affine quantum Schur algebra in the following way (cf. \cite{GV,Lu99}).

\begin{Lem}\cite[3.6.3]{DDF}\label{zr}
For $r\geq0$, there is a surjective algebra homomorphism  $\zr:\dbfHa\twoheadrightarrow\afbfSr$ such that
$$\zr(K\su\bfj)=0(\bfj,r),\;\zr(\ti u_A^+)=A(\bfl,r),\;\;\text{and}\;\;
\zr(\ti u_A^-)=(\tA)(\bfl,r),$$
for all $\bfj\in \afmbzn$ and $A\in \afThnp$, where
$\tA$ is the transpose matrix of $A$ and
$\ti u_A^\pm=\up^{\dim \End(M(A))-\dim M(A)}u_A^\pm$.
\end{Lem}

For convenience, we set $[A]=0\in\afbfSr$ for $A\not\in\afThnr$.
Then we have the following commutation formula in $\afbfSr$.
\begin{Lem}\label{commuting formula zr(t)[diag(la)]}
Let $\la\in\afmbzn$ and $\nu\in\mbz\afPin$. If
$t\in\dbfHa_{\nu}$, then
$$\zr(t)[\diag(\la)]= [\diag(\la+\nu)]\zr(t).$$
\end{Lem}
\begin{proof}
Applying \eqref{[diag(la)][A]} gives
\begin{equation}\label{commuting formula A(bfl,r)[diag(la)]}
C(\bfl,r)[\diag(\la)]=[C+\diag(\la-\co(C))]
=[\diag(\la+\ro(C)-\co(C))]C(\bfl,r)
\end{equation}
for $C\in\afThnpm$ and $\la\in\afmbzn$.
Furthermore by \eqref{bfd(A)}, we conclude that
\begin{equation}\label{deg(uA^+)}
\deg(u_A^+)=\ro(A)-\co(A)\ \text{and}\ \deg (u_A^-)=\co(A)-\ro(A)
\end{equation}
for $A\in\afThnp$.
Combining  \eqref{commuting formula A(bfl,r)[diag(la)]} and
\eqref{deg(uA^+)}
shows that
$\zr(\tu_A^+)[\diag(\la)]=
[\diag(\la+\deg(\tu_A^+))]\zr(u_A^+)$ and
$\zr(\tu_A^-)[\diag(\la)]
=[\diag(\la+\deg(\tu_A^-))]\zr(\tu_A^-).$
This finishes the proof.
\end{proof}
Finally, we prove that the map $\zr:\dbfHa\ra\afbfSr$ induces a natural
algebra homomorphism $\dzr:\ddbfHa\ra\afbfSr$.

\begin{Lem}\label{dzr}
There is an algebra homomorphism $\dzr:\ddbfHa\ra\afbfSr$ such that
$$\dzr(\pi_{\la,\mu}(u))=[\diag(\la)]\zr(u)[\diag(\mu)]$$
for $u\in\dbfHa$ and $\la,\mu\in\afmbzn$.
\end{Lem}
\begin{proof}
Clearly we have
$$[\diag(\la)]\zeta_r\bigg(\sum_{\bfj\in\afmbzn}(\Kbfj-
 \up^{\la\cdot\bfj})\dbfHa+\sum_{\bfj\in\afmbzn}\dbfHa(\Kbfj
 -\up^{\mu\cdot\bfj})\bigg)[\diag(\mu)]=0$$
for $\la,\mu\in\afmbzn$. Thus $\dzr$ is well defined.

Assume $\la',\mu',\la'',\mu''\in\afmbzn$ is such that $\la'-\mu',\la''-\mu''\in\mbz\afPin$.
Let $t\in\dbfHa_{\la'-\mu'}$ and $s\in\dbfHa_{\la''-\mu''}$.
If $\mu'\not=\la''$, then by \eqref{[diag(la)][diag(mu)]},  $\dzr(\pi_{\la',\mu'}(t)\pi_{\la'',\mu''}(s))=0=
\dzr(\pi_{\la',\mu'}(t))\dzr(\pi_{\la'',\mu''}(s))$. If $\mu'=\la''$, then by \eqref{[diag(la)][diag(mu)]} and \ref{commuting formula zr(t)[diag(la)]},  $\dzr(\pi_{\la'\mu'}(t))\dzr(\pi_{\mu'\mu''}(s))
=[\diag(\la')]\zr(t)\zr(s)[\diag(\mu'')]=
\dzr(\pi_{\la'\mu''}(ts))
=\dzr(\pi_{\la'\mu'}(t)\pi_{\mu'\mu''}(s))$.
\end{proof}

Recall that $\ddbfHa$ is a $\dbfHa$-bimodule defined by \eqref{bimodule}.
\begin{Lem}\label{prop of dzr}
We have $\dzr(u_1u_2u_3)=\zr(u_1)\dzr(u_2)\zr(u_3)$ for $u_1,u_3\in\dbfHa$ and $u_2\in\ddbfHa$.
\end{Lem}
\begin{proof}
By \ref{commuting formula zr(t)[diag(la)]}, for $\la,\mu,\nu',\nu''\in\afmbzn$ and
$u'\in\dbfHa_{\nu'}$, $u''\in\dbfHa_{\nu''}$, $u\in\dbfHa_{\la-\mu}$,  we have
$\zr(u')\dzr(\pi_{\la,\mu}(u))\zr(u'')
=\zr(u')\zr(u)[\diag(\mu)]\zr(u'')
=\zr(u')\zr(u)\zr(u'')[\diag(\mu-\nu'')]
=\dzr(\pi_{\la+\nu',\mu-\nu''}(u'uu''))=\dzr(u'\pi_{\la\mu}(u)u'')$.
\end{proof}

\section{The integral form $\ddHa$ of $\ddbfHa$}
Let $\dHap=\spann_\sZ\{\tu_A^+\mid A\in\afThnp\}$,
$\dHam=\spann_\sZ\{\tu_A^-\mid A\in\afThnp\}$, and let
$\dHaz$ be the $\sZ$-subalgebra of $ \dbfHa$
generated by $K_i^{\pm1}$ and $\leb{K_i;0\atop t}\rib$ for $1\leq i\leq n$
and $t>0$, where
$\big[ {K_i;0 \atop t} \big] =\prod_{s=1}^t \frac
{K_i\up^{-s+1}-K_i^{-1}\up^{s-1}}{\up^s-\up^{-s}}.$
Let
$
\dHa=\dHap\dHaz\dHam$ and let
\begin{equation}\label{ddHa}
\ddHa=\spann_{\sZ}\{\tu_A^+1_\la \tu_B^-\mid A,B\in\afThnp,\,\la\in \afmbzn\}\han\ddbfHa,
\end{equation}
where $1_\la=\pi_{\la,\la}(1)$.
Clearly by definition we have
$$\ddHa=\bop_{\la,\mu\in\afmbzn}\pi_{\la,\mu}(\dHa).$$
Furthermore by \eqref{tri deco of dbfHa}, the set
$\{\tu_A^+1_\la \tu_B^-\mid A,B\in\afThnp,\,\la\in \afmbzn\}$ forms a $\sZ$-basis for $\ddHa$.

In \cite[3.8.6]{DDF}, it is conjectured that
$\dHa$ is a $\sZ$-subalgebra of $\dbfHa$.
We will prove in \ref{interal form} that the modified version
$\ddHa$ of $\dHa$ is a $\sZ$-subalgebra of $\ddbfHa$, and
prove in \ref{integral quantum affine Schur Weyl duality} that the restriction of $\dzr:\ddbfHa\ra\afbfSr$ to $\ddHa$ gives a  surjective algebra homomorphism from $\ddHa$ to $\afSr$. Furthermore we will establish certain relation between $\dHa$ and $\ddHa$ in \eqref{relation between dHa and ddHa}.

For $A,B\in\afThnp$ we write
\begin{equation}\label{tuB times tuA}
\tu_B^-\tu_A^+=\sum_{C\in\afThnpm\atop\bfj\in\afmbzn,\,j_n=0}
g_{A,B,C,\bfj}\tu^+_{C^+}\tu^-_{\tr(C^-)}\ti K_1^{j_1}\cdots\ti K_{n-1}^{j_{n-1}}
\end{equation}
where $\tr(C^-)$ is the transpose matrix of $C^-$ and $g_{A,B,C,\bfj}\in\mbq(\up)$

For $A,B\in\afThnp$, $C\in\afThnpm$ and $\la\in\afmbzn$, let
$$f_{A,B,C,\la}=\sum_{\bfj\in\afmbzn,\,j_n=0}g_{A,B,                C,\bfj}
\up^{(\la_1-\la_2)j_1+\cdots+(\la_{n-1}-\la_n)j_{n-1}}.$$
\begin{Lem}\label{key}
Fix $A,B\in\afThnp$ and $\la\in\afmbzn$. Then we have
$f_{A,B,C,\la}\in\sZ$ for $C\in\afThnpm$.
\end{Lem}
\begin{proof}
Let $\sI=\{C\in\afThnpm\mid f_{A,B,C,\la}\not=0\}$. Then $\sI$ is a finite set. It is enough to prove $\sJ:=\{C\in\sI\mid f_{A,B,C,\la}\not\in\sZ\}=\emptyset$. Suppose this is not the case. We choose a maximal element $C_0$ in $\sJ$ with respect to $\pr$. Then $f_{A,B,C_0,\la}\not\in\sZ$. Furthermore, we choose $m>0$ such that $\la+m{\bf 1}\geq\co(C)$ for all $C\in\sI$. Let $\mu=\la+m{\bf 1}\in\afmbnn$, where ${\bf 1}=(\cdots,1,\cdots,1,\cdots)\in\afLann$. Let $r=\sg(\mu)\geq 0$.

Applying \eqref{[diag(la)][diag(mu)]} gives
\begin{equation*}
\begin{split}
\zr(\ti K_1^{j_1}\cdots \ti K_{n-1}^{j_{n-1}})[\diag(\mu)]&=\up^{(\mu_1-\mu_2)j_1+
\cdots+(\mu_{n-1}-\mu_n)j_{n-1}}[\diag(\mu)]\\
&=
\up^{(\la_1-\la_2)j_1+
\cdots+(\la_{n-1}-\la_n)j_{n-1}}[\diag(\mu)].
\end{split}
\end{equation*}
Combining this  with \ref{zr} and \eqref{tuB times tuA} yields
\begin{equation*}
\begin{split}
\zr(\tu_B^-\tu_A^+)[\diag(\mu)]
&=\sum_{C\in\afThnpm\atop\bfj\in\afmbzn,\,j_n=0}
g_{A,B,C,\bfj}C^+(\bfl,r)C^-(\bfl,r)\zr(\ti K_1^{j_1}\cdots
\ti K_{n-1}^{j_{n-1}})[\diag(\mu)]\\
&=
\sum_{C\in\afThnpm}
f_{A,B,C,\la}C^+(\bfl,r)C^-(\bfl,r)[\diag(\mu)].
\end{split}
\end{equation*}
Since $\zr(\tu_B^-\tu_A^+)[\diag(\mu)]=(\tB)(\bfl,r)A(\bfl,r)[\diag(\mu)]\in\afSr$ we conclude that
\begin{equation}\label{Y}
\begin{split}
Y:&=\zr(\tu_B^-\tu_A^+)[\diag(\mu)]-\sum_{C\in\afThnpm\atop C\not\in\sJ}
f_{A,B,C,\la}C^+(\bfl,r)C^-(\bfl,r)[\diag(\mu)]\\
&=\sum_{C\in\sJ}
f_{A,B,C,\la}C^+(\bfl,r)C^-(\bfl,r)[\diag(\mu)]\in\afSr.
\end{split}
\end{equation}
It follows from \ref{tri} that
\begin{equation*}
\begin{split}
Y
&=\sum_{C\in\sJ}f_{A,B,C,\la}\bigg(
C(\bfl,r)+
\sum_{X\in\afThnpm\atop X\p C,\,\bfj\in\afmbzn}h_{C,X,\bfj;r}X(\bfj,r)
\bigg)[\diag(\mu)]\\
&=\sum_{C\in\sJ}f_{A,B,C,\la}\bigg(
\sum_{X\in\afThnpm\atop X\pr C}t_{C,X}[X+\diag(\mu-\co(X))]
\bigg)\\
&=\sum_{X\in\afThnpm}l_{X}[X+\diag(\mu-\co(X))]
\end{split}
\end{equation*}
where $t_{C,C}=1$,
$t_{C,X}=\sum_{\bfj\in\afmbzn}h_{C,X,\bfj;r}
\up^{\bfj\cdot(\mu-\co(X))}$
for  $X\p C$ and $$l_{X}=\sum_{C\in\sJ\atop X\pr C}f_{A,B,C,\la}t_{C,X}.$$
This, together with \eqref{Y} and the fact that the set $\{[T]\mid T\in\afThnr\}$ forms a $\sZ$-basis for $\afSr$, implies that
that $l_X\in\sZ$ for all $X\in\afThnpm$. Thus $f_{A,B,C_0,\la}=l_{C_0}\in\sZ$. This is a contradiction.
\end{proof}

\begin{Thm}\label{interal form}
The $\sZ$-module $\ddHa$ is a $\sZ$-subalgebra of $\ddbfHa$. Thus $\ddHa$ is the integral form for $\ddbfHa$.
\end{Thm}
\begin{proof}
By \eqref{tuB times tuA} and \ref{key} we conclude that
\begin{equation*}
\begin{split}
(\tu_{A_2}^+1_\la\tu_{B_2}^-)(\tu_{A_1}^+1_\mu\tu_{B_1}^-)
&=\sum_{C\in\afThnpm\atop\bfj\in\afmbzn,\,j_n=0}
g_{A_1,B_2,C,\bfj}
(\tu_{A_2}^+1_\la\tu_{C^+}^+)(\tu_{\tr(C^-)}^-
\ti K_1^{j_1}\cdots\ti K_{n-1}^{j_{n-1}}1_\mu\tu_{B_1}^- )\\&=\sum_{C\in\afThnpm}f_{A_1,B_2,C,\mu}
(\tu_{A_2}^+1_\la\tu_{C^+}^+)(\tu_{\tr(C^-)}^-1_\mu\tu_{B_1}^- )\in\ddHa.
\end{split}
\end{equation*}
for $A_1,A_2,B_1,B_2\in\afThnp$ and $\la,\mu\in\afmbzn$. This proves the assertion.
\end{proof}

\begin{Coro}\label{integral quantum affine Schur Weyl duality}
By restriction, the map $\dzr$ defined in \ref{dzr} induces a surjective algebra homomorphism $\dzr:\ddHa\ra\afSr$.
\end{Coro}
\begin{proof}
Applying \ref{prop of dzr} yields
$$\dzr(\ddHa)=\spann_{\sZ}\{A^+(\bfl,r)[\diag(\la)]A^-(\bfl,r)\mid
A\in\afThnpm,\,\la\in\afLanr\}.$$
Combining this with \ref{integral basis for affine Schur algebras} and \ref{interal form} proves the assertion.
\end{proof}

We end this section by studying the relation between $\dHa$ and $\ddHa$. We define the completion algebra $\hddbfHa$ of $\ddbfHa$ as follows.
Let $\hddbfHa$ be
the vector space of all formal (possibly infinite) $\mbq(\up)$-linear combinations
$\sum_{A,B\in\afThnp,\,\la\in\afmbzn}
\beta_{A,B,\la}u_A^+1_\la u_B^-$ satisfying
\begin{itemize}
\item[(F):\,\,]
for any $\mu\in\mathbb Z^n$, the sets
$\{(A,B,\la)\mid A,B\in\afThnp,\,\la\in\afmbzn,\,
\bt_{A,B,\la}\not=0,\,\la-\deg(u_B^-)=\mu\}$ and
$\{(A,B,\la)\mid A,B\in\afThnp,\,\la\in\afmbzn,\,
\bt_{A,B,\la}\not=0,\,\la+\deg(u_A^+)=\mu\}$ are finite.
\end{itemize}
We define the product on $\hddbfHa$ by
$$\sum_{A,B,\la}
\beta_{A,B,\la}u_{A}^+1_\la u_{B}^-\sum_{A',B',\la'}
\ga_{A',B',\la'}u_{A'}^+1_{\la'} u_{B'}^-=\sum_{A,B,\la\atop
A',B',\la'}
\beta_{A,B,\la}\ga_{A',B',\la'}(u_A^+1_\la u_B^-)(u_{A'}^+1_{\la'} u_{B'}^-)$$
where $(u_A^+1_\la u_B^-)(u_{A'}^+1_{\la'} u_{B'}^-)$ is the
product in $\ddbfHa$. Since $(u_A^+1_\la u_B^-)(u_{A'}^+1_{\la'} u_{B'}^-)$ is a linear combination of elements $u_{X}^+1_\mu u_Y^-$ such that $\la+\deg(u_A^+)=\mu+\deg(u_X^+)$ and $\la'-\deg(u_{B'}^-)=\mu-\deg(u_Y^-)$, the right hand side of the above equation is a well defined elements in $\ddbfHa$. In this way,
$\hddbfHa$ becomes an associative algebra.
The element $\sum_{\la\in\afmbzn}1_\la$ is the unit element of $\hddbfHa$.

Similarly, we may define the completion algebra $\hddHa$ of $\ddHa$. Then by definition, $\hddHa$ is the $\sZ$-submodule of $\hddbfHa$ consisting of those elements $\sum_{A,B\in\afThnp,\,\la\in\afmbzn}
\beta_{A,B,\la}u_A^+1_\la u_B^-$ in $\hddbfHa$  with $\bt_{A,B,\la}\in\sZ$ for all $A,B,\la$.

\begin{Prop}\label{Phi}
The map $\Phi:\dbfHa\ra\hddbfHa$ defined by sending $u$ to $\sum_{\la\in\afmbzn}u1_\la$ for $u\in\dbfHa$ is an injective algebra homomorphism.
\end{Prop}
\begin{proof}
Let $u\in\dbfHa_\la$, $w\in\dbfHa_\mu$, where $\la,\mu\in\mbz\afPin$. Since $1_\al w=w1_{\al-\mu}$ for $\al\in\afmbzn$, we have
$$\Phi(u)\Phi(w)=\sum_{\al,\bt\in\afmbzn}(u1_\al) (w1_\bt)=\sum_{\al,\bt\in\afmbzn}uw(1_{\al-\mu}1_\bt)
=\sum_{\bt\in\afmbzn}uw 1_{\bt}=\Phi(uw).$$
Thus $\Phi$ is an algebra homomorphism.

Now let us prove that $\Phi$ is injective. Assume $x=\sum_{A,B\in\afThnp,\,\bfj\in\afmbzn}\bt_{A,B,\bfj}u_A^+\Kbfj u_B^-\in\ker(\Phi)$, where $\bt_{A,B,\bfj}\in\mbq(\up)$. Then
we have
\begin{equation*}
\begin{split}
\Phi(x)&=\sum_{A,B\in\afThnp\atop\bfj,\la\in\afmbzn}\bt_{A,B,\bfj}u_A^+\Kbfj 1_{\la+\deg(u_B^-)}u_B^-\\
&=\sum_{A,B\in\afThnp\atop\la\in\afmbzn}\bigg(
\sum_{\bfj\in\afmbzn}
\bt_{A,B,\bfj}\up^{(\la+\deg(u_B^-))\cdot\bfj}\bigg)u_A^+ 1_{\la+\deg(u_B^-)}u_B^-.
\end{split}
\end{equation*}
This implies that $\sum_{\bfj\in\afmbzn}
\bt_{A,B,\bfj}\up^{(\la+\deg(u_B^-))\cdot\bfj}=0$ for $A,B\in\afThnp$ and $\la\in\afmbzn$. By the proof of \cite[5.2]{DFW05} we see that
$\det(\up^{\mu\cdot\bfj})_{\mu,\bfj\in\mbz_{[a,b]}^n}\not=0$  for   $a<b$, where
$\mbz_{[a,b]}^n=\{\bfx\in\mbz^n\mid a\leq x_i\leq b,\,\text{for},\,1\leq i\leq n\}$. It follows that $\bt_{A,B,\bfj}=0$ for all $A,B,\bfj$ and hence $x=0$. The proof is completed.
\end{proof}
\begin{Rem}\label{Rem}
With the above proposition we may regard  $\dbfHa$ as a subalgebra of $\hddbfHa$. Clearly, we have
\begin{equation}\label{relation between dHa and ddHa}
\dHa\han\hddHa\cap\dbfHa.
\end{equation}
Note that
$\dHa\not=\hddHa\cap\dbfHa.$
For example we have $\frac{K_1K_2}{\up-1}-
\frac{K_1}{\up-1}\not\in\dHa$. But $\frac{K_1K_2}{\up-1}-
\frac{K_1}{\up-1}=\sum_{\la\in\afmbzn}\up^{\la_1}\frac{\up^{\la_2}-1}
{\up-1}1_\la\in \hddHa\cap\dbfHa$. Although
$\dHa\not=\hddHa\cap\dbfHa$, we will prove in \ref{v=1} that the
proper inclusion \eqref{relation between dHa and ddHa} becomes an equality in the classical case.
\end{Rem}


\section{The commutator formulas for $\dbfHa$}
By \cite[1.4.3]{DDF}, $\dbfHa$ is generated by $u^+_{\afE_{i,j}}$, $u^-_{\afE_{i,j}}$ and $K_i$, for $i,j\in\mbz$ and $i<j$. Note
that $M(\afE_{i,j})$ is the indecomposable representation of $\tri$ for all $i<j$.
We derive commutator formulas between indecomposable generators of $\dbfHa$, which will be used in \S6.

For  $A=(a_{i,j})\in \afThnp$, there is a polynomial
$\fka_A=\fka_A(\up^2)\in\sZ$ in $\up^2$ such that, for each finite
field $\field$ with $q$ elements,
$\fka_A|_{\up^2=q}=|\Aut(M_{\field}(A))|$ (see \cite[Cor.~2.1.1]{Peng}). Furthermore, for $A=(a_{i,j})\in \afThnp$, set
$$\fkd(A)=\sum_{i<j,1\leq i\leq n}a_{i,j}(j-i).$$
Then $\dim_\field M(A)=\fkd(A)$ for each finite field $\field$.
For $A,B\in\afThnp$, let
\begin{equation}\label{def of LR}
\begin{split}
L_{A,B}&=\up^{\lan \bfd(B),\bfd(B)\ran}\sum_{A_1,B_1}\vi_{A,B}^{A_1,B_1}
\up^{\lan \bfd(B_1),\bfd(A)+\bfd(B)-\bfd(B_1)\ran}\ti K\su{\bfd(B)-\bfd(B_1)}u_{B_1}^-u_{A_1}^+,\\
R_{A,B}&=
\up^{\lan\bfd(B),\bfd(A)\ran}\sum_{A_1,B_1}
 \ti{\vi_{A,B}^{A_1,B_1}}\up^{\lan \bfd(B)-\bfd(B_1),\bfd(A_1)\ran+\lan \bfd(B),\bfd(B_1)\ran}
 \ti K\su{\bfd(B_1)-\bfd(B)}u_{A_1}^+u_{B_1}^-
\end{split}
\end{equation}
where
\begin{equation*}
\begin{split}
\vi_{A,B}^{A_1,B_1}&=\frac{\fka_{A_1}\fka_{B_1}}{\fka_A\fka_B}\sum_{A_2\in\afThnp}
\up^{2\fkd(A_2)}\fka_{A_2}\vi_{A_1,A_2}^{A}\vi_{B_1,A_2}^B,\\
\ti{\vi_{A,B}^{A_1,B_1}}&=\frac{\fka_{A_1}\fka_{B_1}}{\fka_A\fka_B}
\sum_{A_2\in\afThnp}\up^{2\fkd(A_2)}\fka_{A_2}\vi_{A_2,A_1}^{A}\vi_{A_2,B_1}^B.\\
\end{split}
\end{equation*}
By \cite[2.4.4]{DDF} we have the following result.
\begin{Lem}\label{L=R}
For all $A,B\in\afThnp$,
$L_{A,B}=R_{A,B}$.
\end{Lem}

For $s<t$ we let $$m_{s,t}=\bigg|\bigg\{c\in\mbz\mid 0\leq c\leq \frac{t-s-1}{n} \bigg\}\bigg|-1=\bigg[\frac{t-s-1}{n}\bigg],$$
and set $m_{s,s}=0$.
For $i\in\mbz$, let $\bar i$ denote the integer modulo $n$.

\begin{Lem}\label{vi A,B A1,B1}
Let $A=\afE_{i,j}$, $B=\afE_{k,l}$ with $i<j$ and $k<l$. Assume $A_1,B_1\in\afThnp$.

(1) If  $\vi_{A,B}^{A_1,B_1}\not=0$, then one of the following holds.
\begin{itemize}
\item[(i)] If $A_1=A$ and $B_1=B$, then $\vi_{A,B}^{A_1,B_1}=1$.
\item[(ii)]
If either $A_1\not=A$ or $B_1\not=B$, then $\bar j=\bar l$ and
\begin{equation*}
\vi_{A,B}^{A_1,B_1}=
\begin{cases}
(\up^2-1)\up^{2a_s}&\text{if $A_1=\afE_{i,s}$ and $B_1=\afE_{k,s-j+l}$ for $\max\{i,k-l+j\}<s<j$,}\\
(\up^2-1)^{-1}\up^{2a_i}&\text{if $k-l+j=i$ and $A_1=B_1=0$,}\\
\up^{2a_i}&\text{if $k-l+j<i$, $A_1=0$ and $B_1=\afE_{k,i-j+l}$,}\\
\up^{2a_{k-l+j}}&\text{if $k-l+j>i$, $A_1=\afE_{i,k-l+j}$ and $B_1=0$,}\\
\end{cases}
\end{equation*}
where $a_s=m_{i,s}+m_{k-l+j,s}-m_{i,j}-m_{k,l}+m_{s,j}+j-s,$ for $\max\{i,k-l+j\}\leq s\leq j$.
\end{itemize}

(2) If  $\ti{\vi_{A,B}^{A_1,B_1}}\not=0$, then one of the following holds.
\begin{itemize}
\item[(i)]
If $A_1=A$ and $B_1=B$, then $\ti{\vi_{A,B}^{A_1,B_1}}=1$.
\item[(ii)]
If either $A_1\not=A$ or $B_1\not=B$, then $\bar i=\bar k$ and
\begin{equation*}
\ti{\vi_{A,B}^{A_1,B_1}}=
\begin{cases}
(\up^2-1)\up^{2b_s}&\text{if $A_1=\afE_{s,j}$ and $B_1=\afE_{s,l-k+i}$ for $i<s<\min\{j,l-k+i\}$,}\\
(\up^2-1)^{-1}\up^{2b_j}&\text{if $l-k+i=j$ and $A_1=B_1=0$,}\\
\up^{2b_j}&\text{if $l-k+i>j$, $A_1=0$ and $B_1=\afE_{j,l-k+i}$,}\\
\up^{2b_{l-k+i}}&\text{if $l-k+i<j$, $A_1=\afE_{l-k+i,j}$ and $B_1=0$,}\\
\end{cases}
\end{equation*}
where $b_s=m_{s,j}+m_{s+k-i,l}-m_{i,j}-m_{k,l}+m_{i,s}+s-i,$ for $i\leq s\leq\min\{j,l-k+i\}$.
\end{itemize}
\end{Lem}
\begin{proof}
Applying \cite[(1.2.0.9)]{DDF} yields
$\fka_{\afE_{s,t}}=\up^{2m_{s,t}}(\up^2-1)$ for $s<t$. Furthermore
if $A_1,A_2\not=0\in\afThnp$, then
$$\vi^{\afE_{s,t}}_{A_1,A_2}=
\begin{cases}
1 & \text{ $A_1=\afE_{s,x}$ and $A_2=\afE_{x,t}$ for some $s<x<t$},\\
0 & \text{otherwise}.
\end{cases}
$$
since the module $M(\afE_{s,t})$ is uniserial. Now the assertion follows from the definition of $\vi_{A,B}^{A_1,B_1}$.
\end{proof}

For convenience, we let $\bfd(\afE_{i,i})=0$ for any $i\in\mbz$.
Fix $i<j$ and $k<l$. For $\max\{i,k-l+j\}\leq s\leq j$ let
\begin{equation*}
\begin{split}
f_s&=2a_s+\lan\bfd(\afE_{k,l}),\bfd(\afE_{k,l})\ran+
\lan\bfd(\afE_{k,s-j+l}),\bfd(\afE_{i,j})+
\bfd(\afE_{s,j})\ran\\
\ti f_s&=f_s-f_j=2a_s+\lan\bfd(\afE_{k,s-j+l}),\bfd(\afE_{s,j})\ran
-\lan\bfd(\afE_{s-j+l,l}),\bfd(\afE_{i,j})\ran
\end{split}
\end{equation*}
where $a_s$ is as in \ref{vi A,B A1,B1}. Furthermore
for $i\leq s\leq\min\{j,l-k+i\}$ let
\begin{equation*}
\begin{split}
g_s&=2b_s+\lan\bfd(\afE_{k,l}),\bfd(\afE_{i,j})\ran+\lan\bfd(\afE_{i,s}),
\bfd(\afE_{s,j})\ran+\lan\bfd(\afE_{k,l},\bfd(\afE_{s+k-i,l}))\ran\\
\ti g_s&=g_s-g_i=2b_s+\lan\bfd(\afE_{i,s},\bfd(\afE_{s,j}))\ran-
\lan\bfd(\afE_{k,l}),\bfd(\afE_{k,s+k-i})\ran
\end{split}
\end{equation*}
where $b_s$ is as in \ref{vi A,B A1,B1}. We can now prove the following commutator formulas in $\dbfHa$.

\begin{Prop}\label{commuting formula between elements associated with indecomposable modules}
Assume $i,j\in\mbz$, $i<j$ and $k<l$.

(1) If $\bar j\not=\bar l$ and $\bar i\not=\bar k$, then $u_{\afE_{i,j}}^+u_{\afE_{k,l}}^-=u_{\afE_{k,l}}^-u_{\afE_{i,j}}^+$.

(2) Assume $\bar j=\bar l$ and $\bar i\not=\bar k$.
\begin{itemize}
\item[(i)] If $k-l<i-j$, then
$$u_{\afE_{i,j}}^+u_{\afE_{k,l}}^--u_{\afE_{k,l}}^-u_{\afE_{i,j}}^+
=(\up^2-1)\sum_{i<s<j}\up^{\ti f_s}K_sK_j^{-1}u^-_{\afE_{k,s-j+l}}
u^+_{\afE_{i,s}}+\up^{\ti f_i}K_iK_j^{-1}u_{\afE_{k,i-j+l}}^-.$$
\item[(ii)] If $k-l>i-j$, then
$$u_{\afE_{i,j}}^+u_{\afE_{k,l}}^--u_{\afE_{k,l}}^-u_{\afE_{i,j}}^+
=(\up^2-1)\sum_{k-l+j<s<j}\up^{\ti f_s}K_sK_j^{-1}u^-_{\afE_{k,s-j+l}}
u^+_{\afE_{i,s}}+\up^{\ti f_{k-l+j}}K_kK_j^{-1}u_{\afE_{i,k-l+j}}^+.$$
\end{itemize}
(3) Assume $\bar j\not=\bar l$ and $\bar i =\bar k$.
\begin{itemize}
\item[(i)] If $k-l<i-j$, then
$$u_{\afE_{i,j}}^+u_{\afE_{k,l}}^--u_{\afE_{k,l}}^-u_{\afE_{i,j}}^+
=(1-\up^2)\sum_{i<s<j}\up^{\ti g_s}K_sK_i^{-1}u^+_{\afE_{s,j}}
u^-_{\afE_{s+k-i,l}}-\up^{\ti g_j}K_jK_i^{-1}u_{\afE_{j+k-i,l}}^-.$$
\item[(ii)] If $k-l>i-j$, then
$$u_{\afE_{i,j}}^+u_{\afE_{k,l}}^--u_{\afE_{k,l}}^-u_{\afE_{i,j}}^+
=(1-\up^2)\sum_{i<s<l-k+i}\up^{\ti g_s}K_sK_i^{-1}u^+_{\afE_{s,j}}
u^-_{\afE_{s+k-i,l}}-\up^{\ti g_{l-k+i}}K_lK_i^{-1}u_{\afE_{l-k+i,j}}^+.$$
\end{itemize}
(4) Assume $\bar j =\bar l$ and $\bar i =\bar k$.
\begin{itemize}
\item[(i)]If $k-l=i-j$, then $\afE_{i,j}=\afE_{k,l}$ and
\begin{equation*}
\begin{split}
u_{\afE_{i,j}}^+u_{\afE_{i,j}}^--u_{\afE_{i,j}}^-u_{\afE_{i,j}}^+
&=(\up^2-1)\sum_{i<s<j}(\up^{\ti f_s}K_sK_j^{-1}u^-_{\afE_{i,s}}
u^+_{\afE_{i,s}}-\up^{\ti g_{s}}K_sK_i^{-1}u_{\afE_{s,j}}^+u_{\afE_{s,j}}^-)\\
&\quad\quad+\frac{K_iK_j^{-1}-K_{i}^{-1}K_j}{\up^2-1}\up^{\ti f_i}.
\end{split}
\end{equation*}
\item[(ii)]If $k-l<i-j$, then
\begin{equation*}
\begin{split}
u_{\afE_{i,j}}^+u_{\afE_{k,l}}^--u_{\afE_{k,l}}^-u_{\afE_{i,j}}^+
&=(\up^2-1)\sum_{i<s<j}(\up^{\ti f_s}K_sK_j^{-1}u^-_{\afE_{k,s-j+l}}u^+_{\afE_{i,s}}-\up^{\ti g_s}K_sK_i^{-1}u^+_{\afE_{s,j}}u^-_{\afE_{s+k-i,l}})\\
&\qquad+\up^{\ti f_i}K_iK_j^{-1}u^-_{\afE_{k,i-j+l}}-\up^{\ti g_j}
K_i^{-1}K_ju^-_{\afE_{j+k-i,l}}.
\end{split}
\end{equation*}
\item[(iii)]If $k-l>i-j$, then
\begin{equation*}
\begin{split}
u_{\afE_{i,j}}^+u_{\afE_{k,l}}^--u_{\afE_{k,l}}^-u_{\afE_{i,j}}^+
&=\up^{\ti f_{k-l+j}}K_iK_j^{-1}u^+_{\afE_{i,k-l+j}}+(\up^2-1)\sum_{k-l+j<s<j}\up^{\ti f_s}K_sK_j^{-1}u^-_{\afE_{k,s-j+l}}u^+_{\afE_{i,s}}\\
&\quad-\up^{\ti g_{l-k+i}}K_i^{-1}K_j u^+_{\afE_{l-k+i,j}}-(\up^2-1)
\sum_{i<s<l-k+i}\up^{\ti g_s}K_i^{-1}K_su^+_{\afE_{s,j}}u^-_{\afE_{s+k-i,l}}.
\end{split}
\end{equation*}
\end{itemize}
\end{Prop}
\begin{proof}
For convenience, we let $u^+_{\afE_{s,s}}=
u^-_{\afE_{s,s}}=1$ for $s\in\mbz$.
Let $A=\afE_{i,j}$ and $B=\afE_{k,l}$.
Recall from \eqref{def of LR} the definition of $L_{A,B}$ and $R_{A,B}$. Applying \eqref{deg(uA^+)} gives that
$\ti K^{\bfd(C)}=K^{\deg(u_C^+)}=K^{\ro(C)-\co(C)}$
for $C\in\afThnp$. In particular we have
$\ti K^{\bfd(\afE_{s,t})}=K_sK_t^{-1}$ for $s<t$.
This together with \ref{vi A,B A1,B1} implies that
\begin{equation*}
\begin{split}
L_{A,A}&=\up^{f_j}u_A^-u_A^++(\up^2-1)\sum_{i<s<j}\up^{f_s}K_sK_j^{-1}
u_{\afE_{i,s}}^-u^+_{\afE_{i,s}}+\frac{\up^{f_i}}{\up^2-1}K_iK_j^{-1},\\
R_{A,A}&=\up^{g_i}u_A^+u_A^-+(\up^2-1)\sum_{i<s<j}\up^{g_s}K_sK_i^{-1}
u_{\afE_{s,j}}^+u^-_{\afE_{s,j}}+\frac{\up^{g_j}}{\up^2-1}K_i^{-1}K_j,
\end{split}
\end{equation*}
and if $A\not=B$, then
\begin{equation*}
\begin{split}
L_{A,B}&=
\begin{cases}
\up^{f_j}u_{B}^-u_{A}^+&\!\!\text{if $\bar j\not=\bar l$},\\
\up^{f_j}u_{B}^-u_{A}^+
+(\up^2-1)\sum\limits_{a<s<j}\up^{f_s}K_sK_j^{-1}
u_{\afE_{k,s-j+l}}^-u_{\afE_{i,s}}^++\up^{f_a}K_aK_j^{-1}
u_{\afE_{k,a-j+l}}^-u^+_{\afE_{i,a}}&\!\! \text{if $\bar j=\bar l$},
\end{cases}\\
R_{A,B}&=
\begin{cases}
\up^{f_j}u_{A}^+u_{B}^-&\text{if $\bar i\not=\bar k$},\\
\up^{f_j}u_{A}^+u_{B}^-
+(\up^2-1)\sum\limits_{i<s<b}\up^{g_s}K_sK_i^{-1}
u_{\afE_{s,j}}^+u_{\afE_{s+k-i,l}}^-+\up^{g_b}K_bK_i^{-1}
u_{\afE_{b,j}}^+u^-_{\afE_{b+k-i,l}}& \text{if $\bar i=\bar k$},
\end{cases}
\end{split}
\end{equation*}
where $a=\max\{i,k-l+j\}$ and $b=\min\{j,l-k+i\}$. Combining this with \ref{L=R} proves the assertion.
\end{proof}

\section{The classical ($\up=1$) case}
Let $\afuglq$ be the universal enveloping algebra of  $\afgl$, where
$\afgl={\frak{gl}_n}(\mbq)\ot\mbq[t,t^{-1}]$ is the loop algebra associated to the general linear Lie algebra $\frak{gl}_n(\mbq)$ over $\mbq$.
Using Hall algebras, we will construct the $\mbz$-form $\afuglz$ of $\afuglq$ and prove in \ref{v=1} that $\afuglz$ is a $\mbz$-subalgebra of $\afuglq$. In addition, we will prove in \ref{integral affine Schur Weyl duality} that the natural algebra homomorphism from $\afuglz$ to
$\afSrmbz$  is surjective, where
$\afSrmbz=\afSr\ot_\sZ\mbz$ is the affine Schur algebra over $\mbz$.

Recall the set $\afMnq$ defined in \ref{Notaion}.
We will identify
$\afgl$ with $\afMnq$ via the following lie algebra isomorphism
$$\afMnq\lra\afgl ,\,\,\,\afE_{i,j+ln}\longmapsto E_{i,j}\ot t^l, \,1\le i,j\le n,l\in\mbz. $$
Let $\afuglqp$ (resp., $\afuglqm$, $\afuglqz$) be the subalgebra of $\afuglq$ generated
by $\afE_{i,j}$ (resp., $\afE_{j,i}$, $\afE_{i,i}$), for $1\leq i\leq n$, $j\in\mbz$ and $i<j$.
Then we have
\begin{equation}\label{tri deco}
 \afuglq=\afuglqp\ot\afuglqz\ot\afuglqm,
\end{equation}

Recall $\dHap=\spann_\sZ\{\tu_A^+\mid A\in\afThnp\}$ and
$\dHam=\spann_\sZ\{\tu_A^-\mid A\in\afThnp\}$. Then $\dHap$ and $\dHam$ are all $\sZ$-subalgebras of $\dbfHa$. Note that $\dHap\cong\dHam^{\text{op}}\cong\Ha$, where $\Ha$ is the Hall algebra over $\sZ$ associated with cyclic quivers $\tri$.  Let $\dHapmbq=\dHap\ot_\sZ\mbq$, $\dHapmbz=\dHap\ot_\sZ\mbz$, $\dHammbq=\dHam\ot_\sZ\mbq$ and
$\dHammbz=\dHam\ot_\sZ\mbz$, where
$\mbq$ and $\mbz$ are regarded as $\mbz$-modules by specializing $\up$ to $1$. For $A\in\afThnp$ let $u_{A,1}^+=u_A^+\ot 1$ and $u_{A,1}^-=u_A^-\ot 1$.

\begin{Lem}[{\cite[6.1.2]{DDF}}]\label{th+,th-}
There is a unique injective algebra homomorphism
$\th^+:\dHapmbq\ra\afuglq$ (resp., $\th^-:\dHammbq\ra\afuglq$) taking $u_{{\afE_{i,j},1}}^+\mapsto\afE_{i,j}$ (resp.,$u_{{\afE_{i,j},1}}^-\mapsto\afE_{j,i}$) for all $i<j$ such that $\th^+(\dHapmbq)=\afuglqp$ and $\th^-(\dHammbq)=\afuglqm$.
\end{Lem}

We now use \ref{th+,th-} to
introduce the integral form $\afuglz$ for $\afuglq$.
Let $\afuglzp=\th^+(\dHapmbz)$ and
$\afuglzm=\th^-(\dHammbz)$. Let $\afuglzz$ be the $\mbz$-submodule of $\afuglq$
spanned by $\prod_{1\leq i\leq n}\bpa{\afE_{i,i}}{\la_i}$, for  $\la\in\afmbnn$, where
$$\bpa{\afE_{i,i}}{\la_i}=\frac{\afE_{i,i}(\afE_{i,i}-1)\cdots
(\afE_{i,i}-\la_i+1)}{\la_i!}.$$
Let
\begin{equation*}
\begin{split}
\afuglz&=\afuglzp\afuglzz\afuglzm=\spann_\mbz\bigg\{w_{A}^+\prod_{1\leq i\leq n}\bpa{\afE_{i,i}}{\la_i}w_{B}^-\,\big|\, A,B\in\afThnp,\,\la\in\afmbnn\bigg\},
\end{split}
\end{equation*}
where
$w_A^+=\th^+(u_{A,1}^+)\ \text{and}\ w_B^-=\th^-(u_{B,1}^-).$

\begin{Lem}\label{integral basis for afuglz}
The set $\big\{w_{A}^+\prod_{1\leq i\leq n}\bpa{\afE_{i,i}}{\la_i}w_{B}^-\,\big|\, A,B\in\afThnp,\,\la\in\afmbnn\big\}$
forms a $\mbz$-basis for $\afuglz$ and $\afuglq\cong\afuglz\ot_\mbz\mbq$.
\end{Lem}
\begin{proof}
By \cite{Ko} and \cite[26.4]{Hubk} we conclude that
the set $\big\{\prod_{1\leq i\leq n}\bpa{\afE_{i,i}}{\la_i}\mid\la\in\afmbnn\big\}$ forms a $\mbq$-basis for $\afuglqz$.
Now the assertion follows from \eqref{tri deco} and \ref{th+,th-}.
\end{proof}

To prove that $\afuglz$ is a $\mbz$-subalgebra of $\afuglq$, we need some preparation. Recall $\ddHa$ defined in \eqref{ddHa}.
Let $\ddHambq=\ddHa\ot_\sZ\mbq$. By \ref{interal form}, $\ddHambq$ is a $\mbq$-algebra. For $A\in\afThnp$ and $\la\in\afmbzn$, let $u_{A,1}^+=u_A^+\ot 1$, $u_{A,1}^-=u_A^-\ot 1$ and $1_{\la,1}=1_\la\ot 1$. By \ref{commuting formula between elements associated with indecomposable modules}, we immediately get the following result.

\begin{Lem}\label{commuting formula v=1}
Let $i,j\in\mbz$, $i<j$ and $k<l$. The following formulas hold in $\ddHambq$.

(1) If $\bar j\not=\bar l$ and $\bar i\not=\bar k$, then $1_{\la,1}u_{\afE_{i,j},1}^+u_{\afE_{k,l},1}^-
=1_{\la,1}u_{\afE_{k,l},1}^-u_{\afE_{i,j},1}^+$.

(2) If $\bar j=\bar l$ and $\bar i\not=\bar k$, then
\begin{equation*}
1_{\la,1}(u_{\afE_{i,j},1}^+u_{\afE_{k,l},1}^--u_{\afE_{k,l},1}^-
u_{\afE_{i,j},1}^+)=
\begin{cases}
1_{\la,1}u_{\afE_{k,i-j+l},1}^- &\text{if $k-l<i-j$}\\
1_{\la,1}u_{\afE_{i,k-l+j},1}^+&\text{if $k-l>i-j$}
\end{cases}
\end{equation*}

(3) If $\bar j\not=\bar l$ and $\bar i =\bar k$, then
\begin{equation*}
1_{\la,1}(u_{\afE_{i,j},1}^+u_{\afE_{k,l},1}^--u_{\afE_{k,l},1}^-
u_{\afE_{i,j},1}^+)=
\begin{cases}
-1_{\la,1}u_{\afE_{j+k-i,l},1}^- &\text{if $k-l<i-j$}\\
-1_{\la,1}u_{\afE_{l-k+i,j},1}^+&\text{if $k-l>i-j$}
\end{cases}
\end{equation*}

(4) If $\bar j =\bar l$ and $\bar i =\bar k$, then
\begin{equation*}
1_{\la,1}(u_{\afE_{i,j},1}^+u_{\afE_{k,l},1}^--u_{\afE_{k,l},1}^-
u_{\afE_{i,j},1}^+)=
\begin{cases}
(\la_i-\la_j)1_{\la,1} &\text{if $k-l=i-j$}\\
1_{\la,1}(u^-_{\afE_{k,i-j+l},1}- u^-_{\afE_{j+k-i,l},1})&\text{if $k-l<i-j$}\\
1_{\la,1}(u^+_{\afE_{i,k-l+j},1} - u^+_{\afE_{l-k+i,j},1})&\text{if $k-l>i-j$}
\end{cases}
\end{equation*}
\end{Lem}

Mimicking the construction of $\hddbfHa$,
let $\hddHambq$ be
the $\mbq$-vector space of all formal (possibly infinite) $\mbq$-linear combinations
$\sum_{A,B\in\afThnp,\,\la\in\afmbzn}
\beta_{A,B,\la}u_{A,1}^+1_{\la,1} u_{B,1}^-$ satisfying the property (F) with a similar multiplication. This is an associative $\mbq$-algebra with an identity: $\sum_{\la\in\afmbzn}1_{\la,1}$.
The algebra $\afuglq$ is related to the algebra $\hddHambq$ in the following way (cf. \ref{Phi}).
\begin{Prop}\label{vi}
There is an injective algebra homomorphism $\vi:\afuglq\ra\hddHambq$
such that $\vi(\afE_{i,j})=\sum_{\la\in\afmbzn}u^+_{\afE_{i,j},1}1_{\la,1}$,
$\vi(\afE_{j,i})=\sum_{\la\in\afmbzn}u^-_{\afE_{i,j},1}1_{\la,1}$ and
$\vi(\afE_{i,i})=\sum_{\la\in\afmbzn}\la_i1_{\la,1}$ for $i<j$ and $\la\in\afmbzn$. Furthermore we have $\vi(w_{A}^+)=\sum_{\la\in\afmbzn}u_{A,1}^+1_{\la,1}$ and
$\vi(w_{A}^-)=\sum_{\la\in\afmbzn}u_{A,1}^-1_{\la,1}$ for $A\in\afThnp$.
\end{Prop}
\begin{proof}
For $x,y\in\hddHambq$ we set $[x,y]=xy-yx$.
Then
by \eqref{deg(uA^+)}
we have in $\hddHambq$,
\begin{equation}\label{relation1}
\begin{split}
\bigg[\sum_{\la\in\afmbzn}\la_i1_{\la,1},\sum_{\la\in\afmbzn}
u^+_{\afE_{k,l},1}1_{\la,1}\bigg]&=\sum_{\la\in\afmbzn}\la_iu^+_{\afE_{k,l},1}
1_{\la-\afbse_k+\afbse_l}-\sum_{\la\in\afmbzn}\la_iu^+_{\afE_{k,l},1}1_{\la,1}
\\
&=(\dt_{\bar i,\bar k}-\dt_{\bar i,\bar l})\sum_{\la\in\afmbzn}u^+_{\afE_{k,l},1}1_{\la,1}.
\end{split}
\end{equation}
Applying \cite[6.1.1]{DDF} yields
\begin{equation}\label{relation2}
\begin{split}
\bigg[\sum_{\la\in\afmbzn}u^+_{_{\afE_{i,j},1}}1_{\la,1},
\sum_{\la\in\afmbzn}u^+_{_{\afE_{k,l},1}}1_{\la,1}\bigg]
&=\sum_{\la\in\afmbzn}\big(\dt_{\bar
j,\bar k}u^+_{_{\afE_{i,l+j-k},1}}-\dt_{\bar l,\bar
i}u^+_{_{\afE_{k,j+l-i},1}}\big)1_{\la,1},\\
\bigg[\sum_{\la\in\afmbzn}u^-_{_{\afE_{i,j},1}}1_{\la,1},
\sum_{\la\in\afmbzn}u^-_{_{\afE_{k,l},1}}1_{\la,1}\bigg]&=
\sum_{\la\in\afmbzn}\big(\dt_{\bar i,\bar l}u_{\afE_{k+i-l,j},1}^--\dt_{\bar k,\bar j}
u^-_{\afE_{i+k-j,l},1}\big)1_{\la,1}.
\end{split}
\end{equation}
for $i<j$ and $k<l$.
Furthermore, it is easy to see that $\afuglq$ has a presentation with
generators $\afE_{i,j}$ ($1\leq i\leq n$, $j\in\mbz$), and relations:
\begin{itemize}
\item[(a)]
$[\afE_{i,i},\afE_{k,l}]=(\dt_{\bar i,\bar k}-\dt_{\bar i,\bar l})\afE_{k,l}$.
\item[(b)]
$[\afE_{i,j},\afE_{k,l}]=\dt_{\bar j,\bar
k}\afE_{i,l+j-k}-\dt_{\bar l,\bar i}\afE_{k,j+l-i}$ for $i\not=j$ and $k\not=l$.
\end{itemize}
This, together with \ref{commuting formula v=1}, \eqref{relation1} and \eqref{relation2}, implies that there is an algebra homomorphism $\vi:\afuglq\ra\hddHambq$
defined by sending $\afE_{i,j}$ to $\sum_{\la\in\afmbzn}u^+_{\afE_{i,j},1}1_{\la,1}$, $\afE_{j,i}$ to
$\sum_{\la\in\afmbzn}u^-_{\afE_{i,j},1}1_{\la,1}$, and
$\afE_{i,i}$ to
$\sum_{\la\in\afmbzn}\la_i1_{\la,1}$.

Using an argument similar to the proof of \ref{Phi}, we can show that
there is an injective $\mbq$-algebra homomorphism $\rho^+: \dHapmbq\ra\hddHambq$ (resp., $\rho^-: \dHammbq\ra\hddHambq$) taking $u_{A,1}^+\map\sum_{\la\in\afmbzn}u_{A,1}^+1_{\la,1}$ (resp.,
$u_{A,1}^-\map\sum_{\la\in\afmbzn}u_{A,1}^-1_{\la,1}$) for $A\in\afThnp$. Since $\vi\circ\th^\pm(u^\pm_{\afE_{i,j},1})=\rho^\pm(u^\pm_{\afE_{i,j},1})$ and $\dHapmmbq$ is generated by $u^\pm_{\afE_{i,j},1}$ for $i<j$, we see that $\vi\circ\th^\pm=\rho^\pm$ and hence $$\vi(w_{A}^\pm)=\vi\circ\th^\pm(u_{A,1}^\pm)=
\rho^\pm(u_{A,1}^\pm)=\sum_{\la\in\afmbzn}u_{A,1}^\pm1_{\la,1}$$  for $A\in\afThnp$.

Finally, we prove that $\vi$ is injective.  Assume $$x=\sum_{A,B\in\afThnp,\,\bfj\in\afmbnn}k_{A,B,\bfj}w_A^+\prod_{1\leq i\leq n}(\afE_{i,i})^{j_i}w_{B}^-\in\ker(\vi),$$
where $k_{A,B,\bfj}\in\mbq$.
Then $$\vi(x)=\sum_{A,B\in\afThnp,\,\mu\in\afmbzn}\bigg(\sum_{\bfj\in\afmbnn}
k_{A,B,\bfj}\mu^\bfj\bigg)u_{A,1}^+1_{\mu,1}u_{B,1}^-=0,$$
where $\mu^\bfj=\mu_1^{j_1}\cdots\mu_n^{j_n}$.
This implies that $\sum_{\bfj\in\afmbnn}
k_{A,B,\bfj}\mu^\bfj=0$ for all $A,B,\bfj$. By \cite[6.3.3]{DDF}, we have
$\det(\mu^\bfj)_{\mu,\bfj\in\mbz_l^n}\not=0$ for any $l\geq 1$, where
$\mbz_l^n=\{\la\in\mbz^n\mid 0\leq\la_i\leq l-1,\,\forall i\}$. It follows that
$k_{A,B,\bfj}=0$ for all $A,B,\bfj$ and hence $x=0$. The proof is completed.
\end{proof}

We will identify $\afuglq$ with the subalgebra $\vi(\afuglq)$ of $\hddHambq$ via $\vi$, and hence identify $w_{A}^{\pm}$ with $\sum_{\la\in\afmbzn}u_{A,1}^\pm1_{\la,1}$ for $A\in\afThnp$, etc.

Let$\ddHambz=\ddHa\ot_\sZ\mbz$ and let $\hddHambz$ be the $\mbz$-submodule of $\hddHambq$ consisting of those elements $\sum_{A,B\in\afThnp,\,\la\in\afmbzn}
\beta_{A,B,\la}u_{A,1}^+1_{\la,1} u_{B,1}^-$ in $\hddHambq$  with $\bt_{A,B,\la}\in\mbz$ for all $A,B,\la$. Then by \ref{interal form},
$\hddHambz$ is a $\mbz$-subalgebra of $\hddHambq$.
We can now prove that $\afuglz$ is a $\mbz$-subalgebra of $\afuglq$.

\begin{Thm}\label{v=1}
We have $\afuglz=\hddHambz\cap\afuglq$. In particular, $\afuglz$ is a $\mbz$-subalgebra of $\afuglq$.
\end{Thm}
\begin{proof}
Clearly $\afuglz\han\hddHambz\cap\afuglq$. On the other hand, if
$x\in\hddHambz\cap\afuglq$, then by \ref{integral basis for afuglz}, we may write $$x=\sum_{A,B\in\afThnp\atop\la\in\afmbnn}k_{A,B,\la}
w_{A}^+\prod_{1\leq i\leq n}\bigg({\afE_{i,i}\atop\la_i}\bigg)w_B^-=
\sum_{A,B\in\afThnp\atop \mu\in\afmbzn}
\bigg(\sum_{\la\in\afmbnn}k_{A,B,\la}\bigg({\mu\atop\la}\bigg)\bigg)
u_{A,1}^+1_{\mu,1} u_{B,1}^-,$$
where $k_{A,B,\la}\in\mbq$ and $\big({\mu\atop\la}\big)=
\prod_{1\leq i\leq n}\big({\mu_i\atop\la_i}\big)$.
Since $x\in\hddHambz$ we have $$k_{A,B,\mu}+\sum_{\la\in\afmbnn,\,\sg(\la)<\sg(\mu)\atop
\la_i\leq\mu_i,\,\forall i}
k_{A,B,\la}\bigg({\mu\atop\la}\bigg)
=\sum_{\la\in\afmbnn}k_{A,B,\la}\bigg({\mu\atop\la}\bigg)\in\mbz$$ for $A,B\in\afThnp$ and $\mu\in\afmbnn$, where, as before, $\sg(\mu)=\sum_{1\leq i\leq n}\mu_i$.
Using induction on $\sg(\mu)$, we conclude that $k_{A,B,\mu}\in\mbz$ for all $A,B\in\afThnp$ and $\mu\in\afmbnn$, and hence $x\in\afuglz$. This proves the first assertion. Since
$\hddHambz$ is a $\mbz$-subalgebra of $\hddHambq$ by \ref{interal form}, we conclude that $\afuglz=\hddHambz\cap\afuglq$ is a $\mbz$-subalgebra of $\afuglq$.
\end{proof}

Finally we will establish affine Schur--Weyl reciprocity at the integral level.
Let $\afSrmbq=\afSr\ot_\sZ\mbq$ and $\afSrmbz=\afSr\ot_\sZ\mbz$.
For $A\in\afThnpm$ and $\la\in\afLanr$, let $A(\bfl,r)_1=A(\bfl,r)\ot 1\in\afSrmbq$ and $[\diag(\la)]_1=[\diag(\la)]\ot 1\in\afSrmbq$.
The following result is due to \cite[6.1.3 and 6.1.4]{DDF} (cf. \cite{Yang1}).
\begin{Lem}\label{etar}
There is a surjective algebra homomorphism
$\eta_r:\afuglq\ra \afSrmbq$ such that $\eta_r(\afE_{i,j})=\afE_{i,j}(\bfl,r)_1$ for $i\not=j$ and $\eta_r(\afE_{i,i})=\sum_{\la\in\afLanr}\la_i[\diag(\la)]_1$.
Furthermore we have $\eta_r(w_A^+)=A(\bfl,r)_1$ and $\eta_r(w_A^-)=\tA(\bfl,r)_1$ for $A\in\afThnp$.
\end{Lem}

\begin{Thm}\label{integral affine Schur Weyl duality}
The restriction of $\eta_r$ to $\afuglz$ gives a surjective $\mbz$-algebra homomorphism $\eta_r:\afuglz\twoheadrightarrow\afSrmbz$.
\end{Thm}
\begin{proof}
Since $\eta_r(\afE_{i,i})=\sum_{\mu\in\afLanr}\mu_i[\diag(\mu)]_1$ for $1\leq i\leq n$,
by \eqref{[diag(la)][diag(mu)]}, we conclude that $$\eta_r\bigg(\prod_{1\leq i\leq n}\bigg({\afE_{i,i}\atop\la_i}\bigg)\bigg)=\sum_{\mu\in\afLanr\atop
\mu_i\geq\la_i,\,1\leq i\leq n}\bigg({\mu\atop\la}\bigg)[\diag(\mu)]_1$$
for $\la\in\afmbnn$. It follows that
$\eta_r\left(\prod_{1\leq i\leq n}\left({\afE_{i,i}\atop\la_i}\right)\right)=[\diag(\la)]_1$
for $\la\in\afLanr$.
This, together with \ref{etar} implies that $$\eta_r(\afuglz)=\spann_\mbz\{A^+(\bfl,r)_1[\diag(\la)]_1
A^-(\bfl,r)_1\mid A\in\afThnpm,\,\la\in\afLanr\}.$$
Combining this with  \ref{integral basis for affine Schur algebras} and \ref{v=1} proves the assertion.
\end{proof}


\begin{thebibliography}{99}\frenchspacing
\bibitem{Ar}
S. Ariki, {\em Cyclotomic $q$-Schur algebras as quotients of quantum algebras}, J. Reine Angew. Math. {\bf 513} (1999), 53--69.

\bibitem{ATY}
S. Ariki, T. Terasoma, and H. Yamada, {\em Schur--Weyl reciprocity for the Hecke algebra of $(\mbz/r\mbz)\wr\frak S_n$}, J. Algebra {\bf 178} (1995), 374--390.

\bibitem{CL}
R. Carter and G. Lusztig, \textit{On the modular representations of the general linear and symmetric
groups}, Math. Zeit. {\bf 136} (1974), 193--242.

\bibitem{CP}
C. de Concini and C. Procesi, \textit{A characteristic free approach to invariant theory},
Adv. Math. {\bf 21} (1976), 330--354.


\bibitem{DDF}
B.B. Deng, J. Du and Q. Fu,
{\em A Double Hall Algebra Approach to Affine Quantum Schur--Weyl Theory}, preprint, arXiv:1010.4619.



\bibitem{Donkin}
S. Donkin, \textit{Invariants of several matrices}, Invent. Math. {\bf 110} (1992), 389--401.


\bibitem{Du95}
J. Du, {\em A note on the quantized Weyl reciprocity at roots
of unity}, Alg. Colloq. {\bf 2} (1995), 363--372.



\bibitem{DFW05}
J. Du, Q. Fu and J.-P. Wang, \textit{Infinitesimal quantum
$\frak{gl}_n$ and  little $q$-Schur algebras,} J. Algebra {\bf
287} (2005), 199--233.

\bibitem{DPS}
J. Du, B. Parshall and L. Scott, {\em Quantum Weyl
reciprocity and tilting modules}, Commun. Math. Phys. {\bf 195} (1998), 321--352.

\bibitem{Fu}
Q. Fu, \textit{BLM realization for $\sU_\mbz(\h{\frak{gl}}_n)$}, preprint.

\bibitem{GV}
V. Ginzburg and E. Vasserot, {\em Langlands reciprocity for
affine quantum groups of type $A_n$}, Internat. Math. Res. Notices (1993), 67--85.

\bibitem{Hu}
J. Hu, {\em Schur-Weyl reciprocity between quantum groups and Hecke algebras of type $G(r,1,n)$}. Math. Z. {\bf 238} (2001), 505--521.

\bibitem{Hubk}
J. E. Humphreys, \textit{Introduction to Lie algebras and representation theory}, Springer-Verlag, New York-Berlin, 1978.

\bibitem{Jimbo}
M. Jimbo, \textit{A $q$-analogue of $U(\frak{gl}(N + 1))$, Hecke algebras, and the Yang¨CBaxter equation}, Lett. Math. Phy. {\bf 11}(1986), 247--252.

\bibitem{Ko}
B. Kostant, \textit{Groups over $\bf Z$}, Algebraic Groups and Discontinuous Subgroups, Proc. Sympos. Pure Math. (1966), 90--98.

\bibitem{Lubk}
G. Lusztig, {\em Introduction to quantum groups},
Progress in Math. {\bf 110}, Birkh\"auser, 1993.

\bibitem{Lu99}
G. Lusztig, {\em Aperiodicity in quantum affine $\frak{gl}_n$},
Asian J. Math. {\bf 3} (1999),  147--177.

\bibitem{Peng}
L. Peng, {\em Some Hall polynomials for representation-finite trivial extension algebras},
J. Algebra {\bf 197} (1997), 1--13.


\bibitem{Ri93}
C.~M. Ringel, {\em The composition algebra of a
cyclic quiver}, Proc. London Math. Soc. {\bf 66} (1993), 507--537.

\bibitem{SS}
M. Sakamoto, T. Shoji, {\em Schur--Weyl reciprocity for Ariki--Koike algebras}, J. Algebra {\bf 221} (1999) 293--314.


\bibitem{Weyl}
H. Weyl, {\em The classical groups}, Princeton U. Press, Princeton, 1946.

\bibitem{X97}
J. Xiao, {\em Drinfeld double and Ringel-Green theory of Hall algebras},
J. Algebra {\bf 190} (1997), 100--144.

\bibitem{Yang1}
D. Yang, {\em On the affine Schur algebra of type $A$},
Comm. Algebra {\bf 37} (2009), 1389--1419.

\end{thebibliography}
\end{document}